\documentclass[a4paper,12pt]{article}

\usepackage{xcolor}

\usepackage[margin=3cm,footskip=1cm,  marginparwidth=2.5cm]{geometry} 

\usepackage[T1]{fontenc}

\usepackage{amsmath,amssymb,amsthm,bm,mathtools,xspace,booktabs,url}
\usepackage{graphicx,etoolbox,subfigure} 
\usepackage{microtype} 
\usepackage{marginnote}  

\usepackage[round]{natbib}

\AtBeginEnvironment{thebibliography}{%
  \small
  \setlength\itemsep{0.75em plus 0.5em minus 0.75em}%
}

\usepackage{caption}
\usepackage[raggedright]{titlesec}

\theoremstyle{plain} 
\newtheorem{theorem}{Theorem}

\newtheorem{proposition}[theorem]{Proposition}
\newtheorem{definition}[theorem]{Definition}

\theoremstyle{definition} 

\usepackage{authblk}

\makeatletter
\newcommand\CorrespondingAuthor[1]{%
  \begingroup%
  \def\@makefnmark{}%
  \footnotetext{Corresponding author: #1}%
  \endgroup%
}
\makeatother

\makeatletter
\renewenvironment{abstract}{%
  \small%
  \providecommand\keywords{%
    \par\medskip\noindent\textit{Keywords:}\xspace}%
  \begin{center}%
    \bfseries \abstractname\vspace{-.5em}\vspace{\z@}%
  \end{center}%
  \quote%
}{\endquote}
\makeatother

\usepackage[above,below,section]{placeins}

\usepackage{siunitx}

\usepackage[all]{onlyamsmath}

\usepackage{lipsum}

\usepackage[utf8]{inputenc}
\usepackage{float,multicol,eso-pic,rotating,bbm}

\newcommand{\R}{\mathbb{R}}

\newcommand{\N}{\mathbb{N}}
\newcommand{\Z}{\mathbb{Z}}

\newcommand{\Me}{\mathrm{Me}}

\newcommand{\Mee}{\widehat{\mathrm{Me}}}

\newcommand{\Fe}{\widehat{F}}

\newcommand{\dd}{\mathrm d}
\newcommand{\1}{\mathbf 1}

\newcommand{\bX}{\mathbf{X}}
\newcommand{\bY}{\mathbf{Y}}
\newcommand{\bx}{\mathbf{x}}
\newcommand{\by}{\mathbf{y}}

\newcommand{\bZ}{\mathbf{Z}}

\newcommand{\F}{\mathcal{F}}

\newcommand{\HAC}{{\mathbf{\mbox{\bf Assumption } \mathcal A}}_C}
\newcommand{\HAW}{{\mathbf{\mbox{\bf Assumption } \mathcal A}}_{W_n}}
\newcommand{\HC}{{\mathbf{\mathcal A}}_C}
\newcommand{\HW}{{\mathbf{\mathcal A}}_{W_n}}
\newcommand{\HACbis}{  {\mathbf{\mbox{\bf Assumption } \mathcal A}}_{C_R}}
\newcommand{\HCbis}{  {\mathbf{\mathcal A}}_{C_R}}
\newcommand{\HAmed}{  {\mathbf{\mbox{\bf Assumption } \mathcal A}}_{\mathrm{med}}}
\newcommand{\Hmed}{  {\mathbf{\mathcal A}}_{\mathrm{med}}}

\renewcommand{\P}{\mathrm{P}}

\DeclareMathOperator\E{E}

\DeclareMathOperator\Var{Var}

\usepackage{soul}

\let\oldtextbf\textbf
\renewcommand\textbf[1]{\oldtextbf{\boldmath #1}}

\begin{document}
\title{Standard and robust intensity parameter estimation for stationary determinantal point processes
}

\author[1]{Christophe A.N. Biscio}
\affil[1]{Department of Mathematical Sciences, Aalborg University, Denmark\\
\texttt{christophe@math.aau.dk}.}
\author[2]{Jean-François Coeurjolly}
\affil[2]{Univ. Grenoble Alpes, France\\
\texttt{Jean-Francois.Coeurjolly@upmf-grenoble.fr}.}

\date{}

\maketitle

\begin{abstract}
  This work is concerned with the estimation of the intensity parameter of a stationary determinantal point process. We consider the standard estimator, corresponding to the number of observed points per unit volume and a recently introduced median-based estimator more robust to outliers. The consistency and  asymptotic normality of estimators are obtained under mild assumptions on the  determinantal point process. We illustrate the efficiency of the procedures in a simulation study.

\end{abstract}

\section{Introduction}

Spatial point patterns are datasets containing the random locations of some event of interest which arise in many scientific fields such as biology, epidemiology, seismology and hydrology. Spatial point processes are the stochastic models generating such data. We refer to  \citet{stoyan:kendall:mecke:95}, \cite{illian:et:al:08} or \citet{moeller:waagepetersen:03} for an overview on spatial point processes.  
 The Poisson point process is the reference process to model random locations of points without  interaction. Many alternative models such as Cox point processes (including Neymann-Scott processes, shot noise Cox processes, log-Gaussian Cox processes) or  Gibbs point processes allow us to introduce clustering effects or to produce regular patterns (see again e.g. \citet{moeller:waagepetersen:03} or \citet{illian:et:al:08}). 
First introduced by~\cite{macchi:75}, the interesting class of determinantal point processes has been revisited recently by~\citet{lavancier:moller:rubak:15} in a statistical context. Such processes are in particular designed to model repulsive point patterns. 

In this paper, we focus  on stationary point processes, that is on point processes with distribution invariant by translation, and on first order characteristics for such processes, that is on the intensity parameter denoted by $\lambda$. The nonnegative real parameter $\lambda$ measures the mean number of points per unit volume and is needed for the estimation of second-order characteristics of point processes such as the pair correlation function or the Ripley's K-function, see for instance~\citet{moeller:waagepetersen:03}. Thus, the estimation of $\lambda$ has been the subject of a large literature (see e.g. \cite{illian:et:al:08}). Asymptotic properties for estimators of $\lambda$ are non trivial and may be particularly challenging to obtain for some models such as the class of Gibbs point processes.  

In this paper, we investigate the theoretical and practical properties of two different estimators of $\lambda$ for the class of stationary determinantal point processes.  The first estimator is the standard one, corresponding to the number of observed points  divided by the volume of the observation domain. The second one is a median-based estimator recently proposed by~\citet{coeurjolly:15} to handle outliers such as extra points or missing points.
The form of these two estimators is not novel and follows the aforementioned references. Asymptotic properties for these two estimators have been established under general conditions on the underlying point process. However, these conditions 
have been  checked mainly for Cox processes. We propose two contributions. First, we provide  conditions on the kernel $C$, defining a determinantal point process (see Section~\ref{sec:background} for details), which ensure that the standard and the median-based estimators are consistent and satisfy a central limit theorem. Second, we investigate the finite-sample size properties of the proposed procedures through a simulation study, where, in particular, we evaluate the ability of the estimators to be robust to outliers.

The rest of the paper is organized as follows. A short background on stationary determinantal point processes is presented in Section~\ref{sec:background}. Section~\ref{sec:std} focuses on the standard estimator and details asymptotic properties for this estimator as well as an estimator of its asymptotic variance. Section~\ref{sec:median} deals with the median-based estimator. Finally, we conduct a simulation study in Section~\ref{sec:sim} to compare these estimators in different scenarios. Proofs of the main results are postponed to Appendix.

\section{Stationary Determinantal point processes} \label{sec:background}

\subsection{Background and definition}

For $d\geq1$, let $\bX$ be a spatial point process defined on $\R^d$, which we see as a random locally finite subset of $\R^d$. Let $\mathcal B(\R^d)$ denote the class of bounded Borel sets in~$\R^d$. For $u=(u^1,\dots,u^d)^\top \in \R^d$ and for $A,B\in \mathcal B(\R^d)$, we denote by $|u|=\max_{i=1,\dots,d} |u^i|$ and by  $d(A,B)$ the minimal distance between $A$ and $B$. 

For any $W\in \mathcal B(\R^d)$, we denote by $|W|$ its Lebesgue measure, by $N(\bX \cap W)$ the number of points in $\bX \cap W$ and a realization of $\bX\cap W$ is of the form
$\bx=\{x_1,\dots,x_m\}\subset W$ for some nonnegative finite integer $m$. We consider simple point processes which means that  two points of the process never coincide almost surely. For further details about point processes, we refer to~\cite{daley:vere-jones:03,daley08} and \cite{moeller:waagepetersen:03}. 


The factorial moment measures are quantities of special interest for point processes. 
For any integer $l\ge 1$, $\bX$ is said to have an $l$-th order factorial moment measure $\alpha^{(l)}$ if for all non-negative measurable functions $h$ defined on $\R^{dl}$,
\begin{equation}\label{eq:factorial}
\E \mathop{\sum\nolimits\sp{\ne}}_{u_1,\dots,u_l\in\bX}  h(u_1,\dots,u_l) =
\int_{\R^{dl}} h(u_1,\dots,u_l)\, \alpha^{(l)}( \mathrm du_1\times \cdots \times \mathrm du_l)
\end{equation}
 where the sign $\not=$ over the summation means that $u_1,\dots,u_l$ are pairwise distinct. If $\alpha^{(l)}$ admits a density with respect to the Lebesgue measure on $\R^{dl}$, this density is called the $l$-th order product density  of $\bX$ and is denoted by $\rho_l$.
Note that $\rho_1=\lambda$ and that for the homogeneous Poisson point process $\rho_l(u_1,\dots,u_l)=\lambda^l$. We assume from now on, that $\lambda$ is a positive real number.

The rest of this section is devoted to stationary determinantal point processes on $\R^d$. We refer to~\cite{hough2009zeros} or \cite{lavancier:moller:rubak:15} for a  review  on non-stationary  determinantal point processes on $\mathbb C^d$. 
 
\begin{definition}\label{def:DPP}
Let $C: \R^d \rightarrow \R$ be a function.
A point process $\bX$ on $\R^d$ is  a stationary determinantal point process (DPP) with kernel $C$ and we denote for short $\bX \sim DPP(C)$,  if for all $l\geq 1$ its $l$-th order product density satisfies the relation
\begin{align*}
 \rho_l(x_1,\ldots x_l)=\det [C](x_1,\dots,x_l)
\end{align*}
for  every $(x_1,\dots,x_l)\in \R^{dl}$, where $[C](x_1,\dots,x_l)$ denotes the $l \times l$ matrix with entries $C(x_i-x_j)$,  $1\leq i,j\leq l$.
\end{definition}


Conditions on the kernel $C$ are required  to ensure the existence of $DPP(C)$. To introduce the result, let $S$ be a compact set of $\R^d$ and consider the function from  $S\times S$, 
$(x,y) \mapsto C(x-y)$. By the Mercer theorem (see \cite{riesz1990functionalanalysis}), if we assume that $C$ is continuous, the following series expansion holds
\begin{align}\label{eq:spectral}
 C(x-y) = \sum_{k=1}^{\infty} \beta^S_k \phi_k(x) \overline{\phi_k(y)}
\end{align}
where $\lbrace \phi_k \rbrace_{k\in \N}$ is an orthonormal basis of $L^2(S)$ and where $\beta^S_k$, $k\geq 1$, are real numbers. Finally, let $\mathcal F(h)$ denote the Fourier transform for a function $h\in L^1(\R^d)$ defined for all $t \in \R^d$ by 
\begin{align*}
  \F(h)(t)=\int_{\R^d} h(x) e^{-2i\pi x\cdot t}\dd x,
\end{align*}
a definition which can be extended to  $L^2(\R^d)$ by Plancherel's theorem (see~\cite{stein1971fourier}). The following result gives a sufficient condition to ensure the existence.

\begin{proposition}[{\citet[Proposition 1]{lavancier:moller:rubak:15}, \cite{hough2009zeros}}]\label{DPPexistence}
Assume $C$ is a symmetric continuous real-valued function in $L^2(\R^d)$. Then $DPP(C)$ exists if and only if one of the two statements is satisfied:\\
(i)  For all compact $S\subset \R^d$ and $k\geq 1$, $0\leq \beta_k^S \leq 1$.\\
(ii) $0\leq \F(C)\leq 1$.
\end{proposition}

For most of the kernels, Proposition~\ref{DPPexistence} (ii) provides a more useful way of characterizing existence. To rephrase this condition, any real-valued continuous covariance function $C$ in $L^2(\R^d)$ with $\F(C)\leq 1$ defines a stationary DPP. Some of the results presented hereafter (Propositions~\ref{prop:Rdependent} and~\ref{prop:ZW}) require a slightly more restrictive condition, namely $\mathcal F(C)<1$. To sum up, we consider the following assumption denoted by $\HC$.\medskip

\noindent{$\HAC$}: $C$ is a symmetric  and continuous function, $C \in L^2(\R^d)$, $C(0)=\lambda$ and $0\leq \F(C) <1$.\medskip


\subsection{Mixing-type properties}

We continue this section by discussing Brillinger mixing and $R$-dependence type properties for stationary DPPs. To introduce the first one, we assume that the factorial moment measure exists until a certain $l\geq1$. Then, the  $l$-th order factorial cumulant moment measure $\gamma_{[l]}$ (see~\cite{daley:vere-jones:03,daley08}) is defined
for any  $A_1,\ldots, A_l$ in $\mathcal B(\R^d)$ by
\begin{align*}
\gamma_{[l]}\left( \prod_{i=1}^l A_i \right) = \sum_{j=1}^l (-1)^{j-1} (j-1)! \ \sum_{B_1, \ldots, B_j \in \mathcal{P}_j^l } \prod_{i=1}^j \alpha^{\left(\left|B_i\right|\right)} \left( \prod_{l_i \in B_i} A_{l_i} \right), 
\end{align*}
 where for all $j\leq l$, $\mathcal{P}_j^l$ denotes the set of all partitions of $\lbrace 1,\ldots,l \rbrace$ into $j$ non empty sets $B_1,\ldots,B_j$. For stationary point processes, we define, for $l\geq 2$, the so-called reduced version of the factorial cumulant moment measure  $\gamma^{red}_{[l]}$  by
 \begin{align*}
 \gamma_{[l]} \left( \prod_{i=1}^l A_i \right) = \int_{A_l} \gamma^{red}_{[l]} \left( \prod_{i=1}^{l-1} (A_i-x)  \right) \dd x
 \end{align*}
where $A_1,\ldots, A_l\in \mathcal B(\R^d)$ and  for $i=1,\ldots,l-1$, $A_i-x$ stands for the translation of $A_i$ by $x$. By Hahn-Jordan decomposition (see~\citet[Theorem 5.6.1]{dudley2002real}), we may write $\gamma^{red}_{[l]}=\gamma^{+red}_{[l]} - \gamma^{-red}_{[l]}$ where $\gamma^{+red}_{[l]}$ and $\gamma^{-red}_{[l]}$ are two measures. The total variation measure of $\gamma^{red}_{[l]}$ is then defined as $|\gamma^{red}_{[l]}|=\gamma^{+red}_{[l]} + \gamma^{-red}_{[l]}$. Then, a point process is said to be Brillinger mixing, if for $l\geq 2$,
\begin{align*}
  \left|\gamma^{red}_{[l]}\right| \left( \R^{d(l-1)} \right)  <+\infty.
 \end{align*}
Brillinger mixing is  adapted to DPPs as shown by the following result.
\begin{theorem}[\cite{biscio_Brillinger:2016}]\label{thm:Brillinger}
Any DPP with kernel $C$ verifying  $\HC$ is Brillinger mixing.
\end{theorem}

For $R>0$, $R$-dependence is simpler to define. A point process $\bX$ is  $R$-dependent if for all $A,B \in \mathcal B(\R^d)$ verifying $d(A,B)>R$, $\bX\cap A$ and $\bX \cap B$ are independent.
This criterion  is satisfied by the large subclass of DPPs with compactly supported kernel $C$. 
\begin{proposition} \label{prop:Rdependent}
Let $\bX\sim DPP(C)$ be a DPP with kernel $C$ verifying $\HC$ and such that $C(x)=0$ for $|x|>R$ for some $R>0$. Then, $\bX$ is $R$-dependent.
\end{proposition}



\begin{proof}
Let $A$ and $B$ be two compact subsets in $\R^d$ such that $d(A,B)>R$. We first need to remind briefly the definition of the density for a DPP. More details can be found in \cite{macchi:75} and \cite{lavancier:moller:rubak:15}, particularly in its supplementary materials. 
By assumption $\HC$, $\bX  \cap S \sim DPP(C )  \cap S $ is absolutely continuous with respect to the homogeneous Poisson process on $S$ with unit intensity and has density
 \begin{align}
  f_S(\lbrace x_1,\ldots,x_n \rbrace) = e^{|S|}P(N(\bX \cap S)=0) \ \det [\tilde{C}] (x_1 , \ldots, x_n), \quad x_1,\dots,x_n \in S \label{eq:density}
 \end{align}
 where $\tilde{C}$ is defined by
\begin{align}\label{eq:Ctilde}
 \tilde{C}(x -y) = \sum_{k=1}^{\infty} C_k^S(x-y)
\end{align}
with  $C_1^S(x-y) = C(x-y)$  and $C_k^S(x-y) = \int_S  C_{k-1}^S (x-z) C(z-y) \dd z$ for $k>1$, see Appendix G in the supplementary materials of  \cite{lavancier:moller:rubak:15}.
By induction, if $C$ is compactly supported, so is $\tilde{C}$. Now, let $x_1,\dots, x_p \in A$ and $y_1,\dots,y_q\in B$, for $p,q\geq 1$. Since $\tilde{C}(x) = 0$ for $|x|>R$,   $[\tilde{C}] (x_1,\ldots,x_p,y_1,\ldots,y_q)$ is a block diagonal matrix. Then, by applying~\eqref{eq:density} with  $S = A \cup B$, it is straightforwardly seen that
 \begin{align*}
  f_S(\lbrace x_1,\ldots,x_p,y_1,\ldots,y_q \rbrace) \propto f_A(\lbrace x_1,\ldots,x_p \rbrace ) f_B(\lbrace y_1,\ldots,y_q \rbrace )
 \end{align*}
where the normalizing constant is determined by the condition $\int f_S = 1$, whereby we deduce the result.
\end{proof}

For some statistical applications, another type of mixing coefficient often used is the $\alpha$-mixing coefficient defined as follows for spatial point processes (see e.g.\ \cite{politis:98}): let $j,k\geq 1$ and $m>0$
\begin{align}
  \alpha_{j,k}(m)=\sup \{&  
|P(A\cap B) - P(A)P(B)|:\, A\in \mathcal{F}(\Lambda_1),\, B\in \mathcal{F}(\Lambda_2), \nonumber\\
& \Lambda_1\in \mathcal B(\R^d),\,\Lambda_2 \in \mathcal B(\R^d),\, |\Lambda_1|\leq j,\, |\Lambda_2|\leq k,\, d(\Lambda_1,\Lambda_2)\geq m
  \} \label{def:mixing}
\end{align}
where $\mathcal{F}(\Lambda_i)$ is the $\sigma$-algebra generated by
$\bX\cap \Lambda_i$, $i=1,2$. It is still an open question to know whether there are general conditions on the kernel $C$ of a DPP providing a control of $\alpha$-mixing coefficients. However, an obvious consequence of Proposition~\ref{prop:Rdependent} is that for DPPs with compactly supported kernel, $\alpha_{j,k}(m)=0$ for any $m>R$ and any $j,k\geq 1$.

\subsection{On the distribution of the number of points}

For general point processes, it is sometimes not easy to see what the distribution of the number of points in a compact set $S$ is. For DPPs, we can actually show that this distribution is, for large compact set $S$, quite close to the probability distribution of a Poisson variable. This interesting behaviour will be exploited in Section~\ref{sec:median}. We let $\Pi(\theta)$ denote a Poisson random variable with parameter $\theta$. The following proposition is based on results obtained by~\cite{zacharovas:hwanf:2010}.

\begin{proposition}\label{prop:ZW}
  Let $\bX$ be a DPP with kernel $C$ verifying $\HC$. Define for any $m\geq 0$
\begin{align*}
d_0(S,m) &=  \P(N(\bX\cap S)=m) - \P(\Pi(\lambda |S|)=m)  \\
d_1(S,m) &= \P(N(\bX\cap S)=m) - \P(\Pi(\lambda |S|)=m) \left( 1-|S|\omega(m,\lambda|S|) \check C_0/2 \right)
\end{align*}
where $\check C_0=\int_{\R^d}C^2(x)\dd x$ and $\omega(m,\ell) = ((m-\ell)^2-m)/\ell^2$ for any $m\geq 0$ and $\ell>0$. Then, there exists three constants  $\kappa_0,\kappa_1$ and $\kappa_1^\prime$, independent of $m$, such that  for all compact $S\subset \R^d$ we have
\begin{equation}\label{eq:d0d1}
  |d_0(S,m)| \leq \frac{\kappa_0}{\sqrt{|S|}} \qquad \mbox{ and } \qquad |d_1(S,m)| \leq \frac{\kappa_1}{\sqrt{|S|}} + \frac{\kappa_1^\prime}{|S|}.
\end{equation}
In particular, 
\begin{equation}\label{eq:k0k1}
  \kappa_0 = \sqrt{3}(\sqrt{e}-1) \frac{ \check C_0\sqrt{\lambda}}{(\lambda- \check C_0)^2} \qquad \mbox{ and } \qquad
  \kappa_1  = \frac{\sqrt{15}(\sqrt{e}-1)}2 \frac{  \check C_0^2 \sqrt{\lambda}}{(\lambda- \check C_0)^3} . 
\end{equation}
\end{proposition}

\section{Estimators of $\lambda$}

We are interested in the estimation of the intensity parameter $\lambda$ based on a single realization of a DPP, $\bX$, observed on an increasing sequence of bounded domains $W_n \subset \R^d$. The standard estimator  is considered in Section~\ref{sec:std} while the  median-based estimator is studied in Section~\ref{sec:median}.

\subsection{Standard estimator of $\lambda$} \label{sec:std}

In this section, we assume  that
$\lbrace W_n \rbrace_{n\geq 1}$ is a sequence of bounded convex subsets of $\R^d$ such that for all $n\geq 1$, $W_n \subset W_{n+1}$ and there exists an Euclidean ball  included  in $W_n$ with radius denoted by $r(W_n)$ tending to infinity as $n$ tends to infinity. To shorten, such a sequence $\lbrace W_n \rbrace_{n\in \N}$ is said to be regular.

For a point process $\bX$, the standard estimator of $\lambda$  is given by
 \begin{align}\label{eq:std}
 \widehat{\lambda}_n^{\mathrm{std}} =\frac{N(\bX \cap W_n)}{|W_n|} .
\end{align}
Further, it is well-known that if the stationary point process is ergodic, a property established by \cite{Soshnikov:00} for stationary DPPs, then this unbiased estimator is strongly consistent as $n\to \infty$. Using the Brillinger mixing property (Theorem~\ref{thm:Brillinger}), we obtain the following result.

\begin{proposition}[\cite{biscio_Brillinger:2016,SoshnikovGaussianLimit}]\label{prop:std}
Let $\bX$ be a stationary DPP with pair correlation function $g$,  kernel $C$ verifying $\HC$  and  let $\lbrace W_n\rbrace_{n\in \N}$ be a regular sequence of subsets of $\R^d$. Then, as $n\to \infty$
\begin{align*}
\sqrt{|W_n|} \left(  \widehat \lambda_n^{\mathrm{std}}- \lambda \right) \to  \mathcal N(0,\sigma^2)
\end{align*}
in distribution, where  $\sigma^2=   \lambda + \lambda^2 \int_{\R^d} (g(w) - 1)dw$. In particular,  for $\bX\sim DPP(C)$, we have  $\sigma^2= \lambda -\check C_0= \lambda -\int_{\R^d} C(x)^2 \dd x$.
\end{proposition}

The last result is not restricted to  DPPs and is valid for a lot of spatial point processes (including some Cox processes, Gibbs point processes,\dots), up to the form of $\sigma^2$ that is in general known only in terms of the pair correlation function. 
The estimation of $\sigma^2$  has therefore been an important topic. We refer the reader to \cite{heinrich:prokevsova:10} for a discussion of this challenging topic. In the latter paper,  the following estimator is proposed.
\begin{align*}
 \widehat{\sigma}_n^2 = \widehat{\lambda}_n^{\mathrm{std}} + \mathop{\sum\nolimits\sp{\ne}}_{x,y \in \bX \cap W_n} \frac{k\left(\frac{y-x}{|W_n|^{1/d}b_n}\right)  }{|(W_n-x) \cap (W_n-y)  |} - |W_n|b_n^d \widehat \lambda_n^{\mathrm{std}}(\widehat \lambda_n^{\mathrm{std}}-|W_n|^{-1})  \int_{W_n} k(x) \, \mathrm dx
\end{align*}
  where $k:\R^d \to [0,\infty)$ plays the role of a kernel and  $\{b_n\}_{n\geq 1}$ is a sequence of real numbers playing the role of a bandwidth. \cite{heinrich:prokevsova:10} obtained in particular the following result which can be directly applied to DPPs with kernel $C$ satisfying $\HC$.

\begin{proposition}[\cite{heinrich:prokevsova:10}]
Let $k:\R^d\to [0,\infty)$ be a symmetric, bounded and continuous function at the origin of $\R^d$ such that $k(0)=1$. Let $\{W_n\}_{n\geq 1}$ be a sequence of regular subsets of $\R^d$ and $\{b_n\}_{n\geq 1}$ be a sequence of real numbers such that as $n\to \infty$
\[
  b_n \to 0, \quad b_n^2 |W_n|^{1/d} \to 0 \quad \mbox{ and } \quad b_n |W_n|^{1/d}r(W_n)^{-1}\to 0
\]
where $r(W_n)$ stands for the inball radius of $W_n$. If $\bX$ is a Brillinger mixing point process, then as $n\to \infty$, $\widehat \sigma_n^2 \to \sigma^2$ in $L^2$.
\end{proposition}

\subsection{Median-based estimator of $\lambda$} \label{sec:median}

For any real-valued random variable $Y$, we denote by $F_Y(\cdot)$ its cdf, by $F_Y^{-1}(p)$ its quantile of order $p\in (0,1)$ and by $\Me_Y=F_Y^{-1}(1/2)$ its theoretical median. Based on a sample $\mathbf{Y}=(Y_1,\dots,Y_n)$ of $n$ identically distributed random variables, we denote by $\Fe(\cdot;\mathbf Y)$ the empirical cdf, and by $\Fe^{-1}(p;\mathbf Y)$ the sample quantile of order $p$ given by
\begin{equation} \label{def:quantile}
  \Fe^{-1}(p;\mathbf Y) = \inf \{ x\in \R: p \leq \Fe(x;\bY)\}.
\end{equation}
The sample median is simply denoted by $\Mee(\mathbf Y)=\Fe^{-1}(1/2;\mathbf Y)$.

In this section, we assume the following assumption, denoted by $\HW$ for the sequence of bounded domain $\{W_n\}_{n\geq 1}$. \medskip

\noindent{$\HAW$}: The domain of observation $W_n$ can be decomposed as $W_n= \cup_{k\in \mathcal K_n} C_{n,k}$ where the cells $C_{n,k}$ are non-overlapping and equally sized with  volume $c_n=|C_{n,k}|$ and where $\mathcal K_n$ is a subset of $\Z^d$ with cardinality $k_n=|\mathcal K_n|$. As $n\to \infty$, $k_n\to\infty$, $c_n\to \infty$.

The standard estimator of $\lambda$ is given by~\eqref{eq:std}. To define a more robust one, we can note that
\begin{equation}\label{eq:mean}
  \widehat\lambda_n^{\mathrm{std}} = \frac1{k_n} \; \sum_{k\in \mathcal K_n} \frac{N(\bX \cap C_{n,k})}{c_n} 
\end{equation}
since $|W_n|=k_n c_n$, i.e. $\widehat \lambda_n^{\mathrm{std}}$ is nothing else than the sample mean of intensity estimators computed in cells $C_{n,k}$. The strategy adopted by \cite{coeurjolly:15} was to replace the sample mean by the sample median, which is known to be more robust to outliers. Quantile estimators based on count data or more generally on discrete data can cause some troubles in the asymptotic theory (see e.g. \cite{david:nagaraja:03}). To bypass the discontinuity problem of the count variables $N(\bX \cap C_{n,k})$, we follow a well-known technique (e.g. \citet{machado:silva:05}) which introduces smoothness. Let $(U_k, k\in \mathcal K_n)$ be a collection of independent and identically distributed  random variables, distributed as $U\sim \mathcal U([0,1])$. Then, for any $k\in \mathcal K_n$, we define 
\begin{equation}\label{eq:Znk}
  Z_{n,k}  = N(\bX \cap C_{n,k}) + U_k \quad \mbox{ and } \quad 
  \bZ=(Z_{n,k}, \; k\in \mathcal K_n).
\end{equation}
By $\HAW$ and the stationarity of $\bX$, the variables $Z_{n,k}$ are identically distributed and we let $Z\sim Z_{n,k}$. 
The jittering effect shows up right away: the cdf of $Z$ is given for any $t\geq 0$ by
\[
  F_{Z}(t) = P( N(\bX \cap C_{n,0})\leq  \lfloor t \rfloor -1) + P(N(\bX \cap C_{n,0})=\lfloor t\rfloor ) \,(t-\lfloor t \rfloor),
\]
and is continuously differentiable whereby we deduce that $Z$ admits a density $f_Z$ at $t$ given by
  $f_Z(t) = P(N(\bX \cap C_{n,0})=\lfloor t\rfloor )$.
We  define the jittered median-based estimator of $\lambda$ by
\begin{equation} \label{eq:median}
  \widehat \lambda_n^{\mathrm{med}} =\frac{\Mee (\bZ)}{c_n}
\end{equation}
where the sample median is defined by~\eqref{def:quantile}. To derive asymptotic properties for $\widehat\lambda_n^{\mathrm{med}}$, we need to consider the subclass of compactly supported DPPs, summarized by the following assumption. \medskip

\noindent $\HACbis$. The kernel $C$ of the stationary DPP satisfies $\HC$. There exists $R>0$ such that $C(x)=0$ for any $|x|>R$. \medskip

Finally a technical condition, ensuring the asymptotic positivity of the density at the median is required. \medskip

\noindent $\HAmed$. $\liminf_{n\to \infty} s_n>0$, where $s_n=\sqrt{c_n}\P(N(\bX \cap C_{n,0})=\lfloor \Me_Z\rfloor)$.

\begin{proposition} \label{prop:median}
Assume that the sequence of domains satisfies $\HW$, that the DPP with kernel $C$, $\bX$, satisfies $\HCbis$ and that $\Hmed$ holds. Then, as $n\to \infty$,\\
(a) \\
\begin{equation}\label{eq:resa}
  \sqrt{|W_n|}s_n \left( \frac{\Mee (\bZ)}{c_n} - \frac{\Me_Z}{c_n}\right) \to \mathcal N(0,1/4)
\end{equation}
in distribution.\\
(b)  $\Me_Z-\lambda c_n = o(\sqrt{c_n})$.\\
(c) If in addition, $\sqrt{k_n}(\Me_Z-\lambda c_n)/\sqrt{c_n} \to 0$ as $n\to \infty$, then
\begin{equation}\label{eq:resb}
  2\sqrt{|W_n|} s_n\left( \widehat \lambda_n^{\mathrm{med}} - \lambda \right) \to \mathcal N(0,1)
\end{equation}
in distribution.
\end{proposition}

It is worth comparing \eqref{eq:resb} with \citet[Corollary~6]{coeurjolly:15} devoted to Cox processes. For Cox processes, the author is able to provide conditions on the Cox process for which $s_n$ admits a limit equal to $(2\pi \sigma^2)^{-1/2}$, where $\sigma^2=\lambda+\int_{\R^d}(g(x)-1)\dd x$. To prove this, \citet{coeurjolly:15} uses explicitly the connection between Poisson and Cox point processes. We were unable to prove any result of that type for stationary DPP (and actually conjecture that there  is no limit). An interesting fact can however be noticed. Since DPP are purely repulsive models, the asymptotic variance of $\widehat \lambda^\mathrm{std}$, that is $\sigma^2$ is always bounded by $\lambda$, that is, by the corresponding asymptotic variance under the Poisson model. Transferring this to the median-based estimator, we conjecture that $|W_n|\Var(\widehat \lambda^\mathrm{med})$ is asymptotically bounded by the corresponding variance under the Poisson case, which is precisely $\pi \lambda/2$. By replacing $\lambda$ by its estimate, we have the basis to propose an asymptotic conservative confidence interval for $\lambda$.

We end this section by stating a sufficient condition which ensures $\Hmed$.

\begin{proposition}
If the DPP with kernel $C$, $\bX$, satisfies $\HC$, then $\Hmed$ holds if
\begin{equation}
  \label{eq:conditionAmed}
  \max \left( 
(2\pi\lambda)^{-1/2} - \kappa_0 , (2\pi\lambda)^{-1/2} \left( 1+ \frac{\check C_0}{2\lambda}\right)-\kappa_1
  \right) >0
\end{equation}
where $\kappa_0$ and $\kappa_1$ are given by~\eqref{eq:k0k1} and  $\check C_0=\int_{\R^d} C^2(x)\dd x$.
\end{proposition}

Condition~\eqref{eq:conditionAmed} is a theoretical condition, which allows us to understand what kind of kernels $C$ can satisfy $\Hmed$. From a practical point of view, if we assume that the data can be modelled by a stationary DPP, the condition~\eqref{eq:conditionAmed} can be tested by plugging an estimate of $\lambda$ and $\check C_0$. Note that $\check C_0= \lambda-\sigma^2$ can be  estimated by $\widehat \lambda_n^\mathrm{med}-\widehat \sigma_n^2$ where $\widehat \sigma_n^2$ is the estimate detailed in Section~\ref{sec:std} into which we can also replace $\widehat \lambda_n^\mathrm{std}$ by $\widehat \lambda_n^\mathrm{med}$ if the presence of outliers is suspected.

\begin{proof} 
Using Proposition~\ref{prop:ZW} (and the notation therein), we have
\[
  s_n = \sqrt{c_n} \P(N(\bX \cap C_{n,0})=\lfloor \Me_Z\rfloor) \geq \max \left(L_{0,n} , L_{1,n}\right)
\]
where 
\begin{align*}
L_{0,n}=& \sqrt{c_n}\P( \Pi(\lambda c_n)=\lfloor \Me_Z \rfloor) -\kappa_0 \\
L_{1,n}=& \sqrt{c_n}\P(\Pi(\lambda c_n)= \lfloor \Me_Z \rfloor) \left( 1-c_n\omega(\lfloor\Me_Z\rfloor,\lambda c_n)\check C_0/2\right) - \kappa_1.
\end{align*}
As $n\to \infty$, Proposition~\ref{prop:median} (a) and~\eqref{eq:stirling} yield that 
\[
  \P( \Pi(\lambda c_n)=\lfloor \Me_Z \rfloor) \to (2\pi\lambda)^{-1/2} 
  \quad \mbox{ and }\quad
  c_n \omega(\lfloor\Me_Z\rfloor,\lambda c_n) \to -1
\]
whereby we deduce the result.  
\end{proof} 


\section{Simulation study} \label{sec:sim}

In this section, we investigate the performances of~\eqref{eq:std} and~\eqref{eq:median} for planar stationary DPPs. The study follows the one done in~\cite{coeurjolly:15} which was mainly designed for Cox processes.

We consider the following DPP model introduced by~\cite{biscio:lavancier:15}. For $\nu>0$, let $j_\nu$ be the first positive zeros of the Bessel function of the first kind $J_\nu$ and define the constant $M$ by
$M \lambda^{1/d}= \left( 2^{d-2} j^2_{\frac{d-2}{2}} \Gamma\left(\frac{d}{2} \right)\right)^{1/d} /\pi^{1/2} $.
When $d=2$, $M \lambda^{1/2} = j_0/\pi^{1/2} \approx 1.357$. Now, let $R \in (0,M]$. We define the kernel $C_R=u_R \ast u_R$ where
\begin{align*}
 u_R(x)=\kappa \ \frac{J_{\frac{d-2}{2}}\left(2j_{\frac{d-2}{2}}\frac{|x|}{R}\right)}{|x|^\frac{d-2}{2}}\ \1_{\left\lbrace |x| <\frac{R}{2} \right\rbrace},
\end{align*}
and $\kappa^2=\frac{4 \Gamma\left(d/2\right)}{ \lambda\pi^{d/2} R^2  }\left(J'_{\frac{d-2}{2}}(j_{\frac{d-2}{2}})\right)^{-2}$. Note that there exist many other kernels of DPPs which are compactly  supported, see for instance~\citet[Proposition 4.1]{biscio:lavancier:15}. The advantage of the kernel $C_R$ is that its Fourier transform is explicit and thus $\HCbis$ can be investigated.
In particular, for any, $x\in \R^d$, it can be shown that, $\F(C_R)(x) \leq \frac{R^d}{M^d}$. Thus, the kernel $C_R$ satisfies $\HCbis$ for all $R<M$. Two different versions of this model, denoted {\sc dpp1} and {\sc dpp2}, obtained by setting $R$ to the values $R=M/4$ and $R=3M/4$ respectively, are considered in the simulation study. Figure~\ref{fig:ex} depicts the pair correlation functions $g$ for these models as well as a realization of each of these processes. It is to be noted that the models {\sc dpp1} and {\sc dpp2} satisfy $\HCbis$ and $\Hmed$. Specifically, the constants involved in~\eqref{eq:conditionAmed} are numerically evaluated to 0.057 and 0.021 respectively.\\

\begin{figure}[htbp]
 \begin{multicols}{3} 
 \hspace*{-.5cm}\includegraphics[scale=.35]{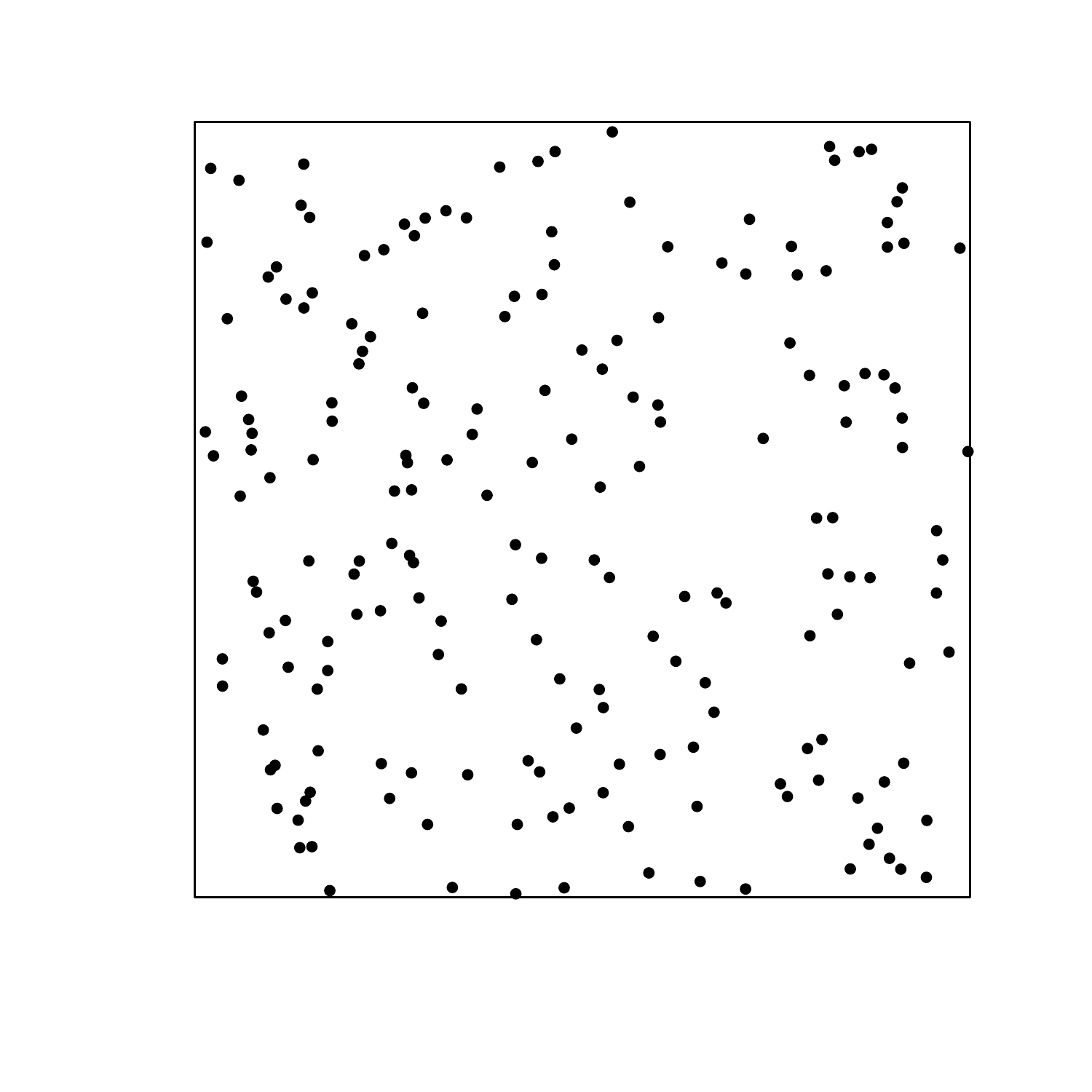}
 \hspace*{-.5cm}\includegraphics[scale=.35]{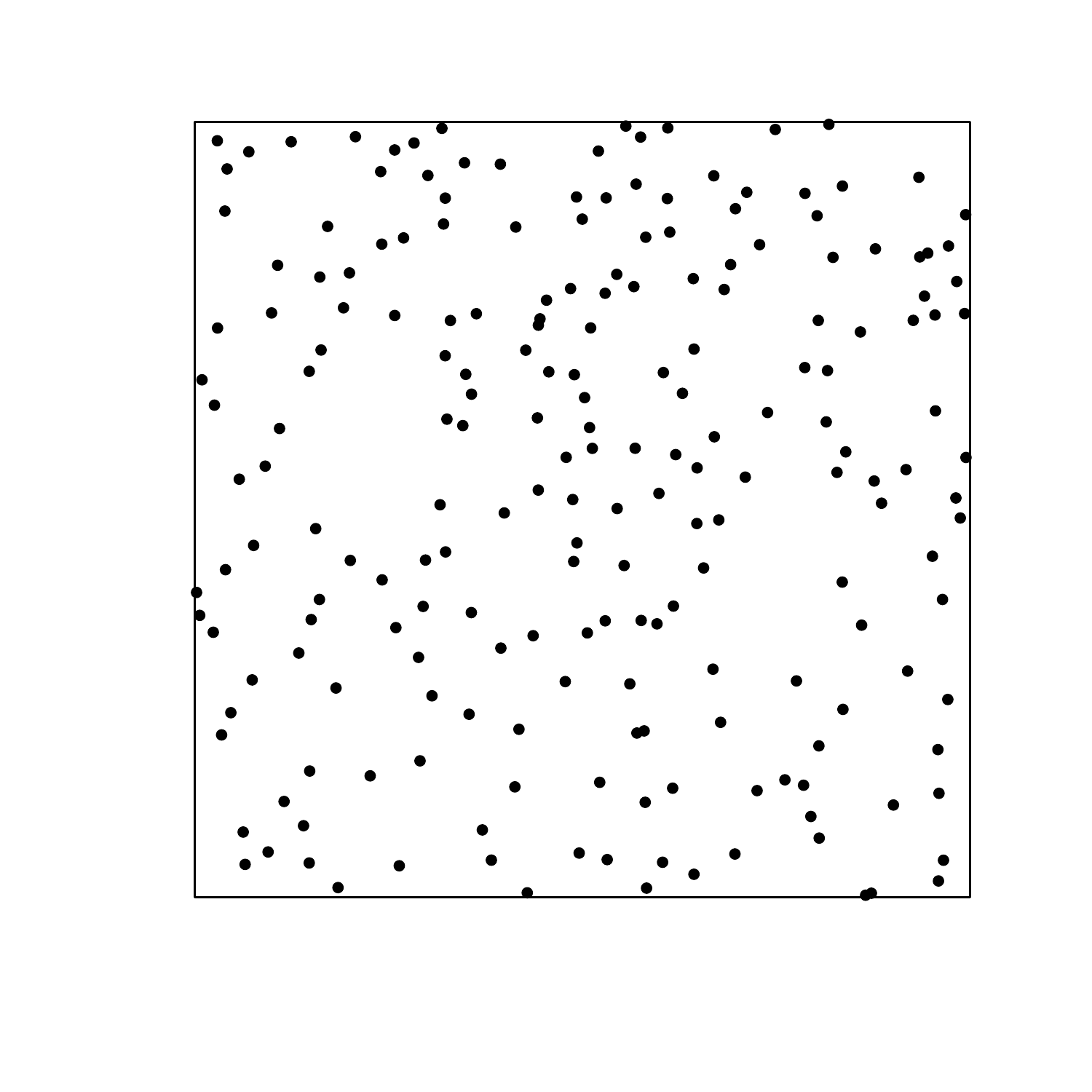}
 \hspace*{0cm}\includegraphics[scale=.33]{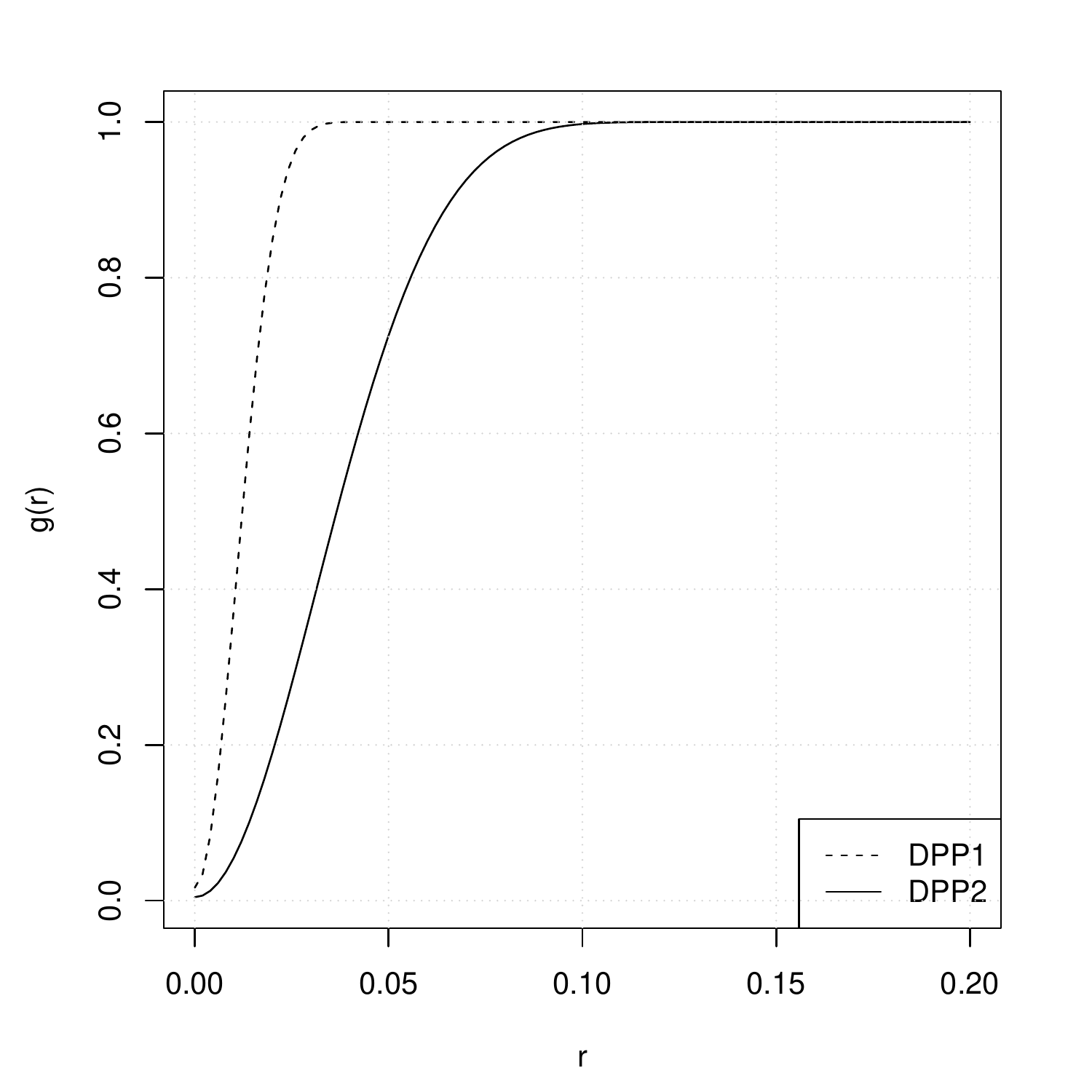}
 \end{multicols}
\caption{\label{fig:ex} Left (resp. middle): Realization of the model {\sc dpp1} (resp. {\sc dpp2}) on the domain $[-1,1]^2$. The intensity equals $\lambda=50$; Right: Pair correlation function $g$ for the models {\sc dpp1,dpp2}. }
\end{figure}

The models {\sc dpp1, dpp2} are generated on  $W_n=[-n,n]^2$ for $n=1,2$ and we  consider the three following settings: let $\by$ be a realization from one of the two models described above, generated on $W_n$ and with $m$ points. The observed point pattern is denoted by $\bx$ and is obtained as follows.
\begin{itemize}
  \item[(A)] Pure case: no modification is considered, $\bx=\by$.
  \item[(B)] A few points are added: in a sub-square $\Delta_n$ with side-length $|\Delta_n|^{1/2}=n/5$ included in $W_n$ and randomly chosen, we  generate a point process $\by^{\mathrm{add}}$ of $n^{\mathrm{add}}= \rho \, m$ uniform points in $\Delta_n$. We choose $\rho=0.05$ or $0.1$. Then, we define $\bx=\by \cup \by^{\mathrm{add}}$.
  \item[(C)] A few points are deleted: let $\Delta_n$ be a randomly chosen sub-square included in $W_n$. The volume of $\Delta_n$ is chosen such that $\E(N({\bX \cap \Delta_n})) = \rho \, \E(N(\bX \cap W_n)) = \rho \lambda |W_n|$, with $\bX \sim  \textsc{ dpp1}$ or {\sc dpp2}, and we choose either $\rho=0.05$ or $0.1$. Then, we define $\bx= \by\setminus {\Delta_n}$, i.e. $\bx$ is the initial configuration thinned by 5\% or 10\% of its points (on average) located in the sub-square $\Delta_n$.
\end{itemize}
An illustration of settings (B) and (C) is proposed in Figure~\ref{exContamination}.
We conduct a Monte Carlo simulation and generate $500$ replications of the models {\sc dpp1, dpp2} with intensity parameter $\lambda=50$ and for the three different settings (A)-(C). For each replication, we evaluate $\widehat\lambda_n^{\mathrm{std}}$ and $\widehat\lambda_n^{\mathrm{med}}$ for different number of non-overlapping and equally sized cells $k_n$. More precisely, we choose $k_n=9,16,25,36,49$. The empirical results can be sometimes quite influenced by the choice of the number of blocks $k_n$. In a separate analysis not reported, we have noticed that, depending on the situation, the estimates could be far from $\lambda$ for some $k_n$ but, also, that there are  consecutive values of $k_n$ producing close values. Following this empirical finding, we propose the data-driven estimator for $\lambda$, denoted by $\widetilde \lambda_n^\mathrm{med}$ and  defined as the median of the median-based estimators, that is 
\begin{equation}
  \label{eq:medianDD}
  \widetilde \lambda_n^{\mathrm{med}} = \Mee \left( \left\{ \widehat \lambda_n^{\mathrm{med}}, \; k_n=9,16,25,36,49\right\} \right).
\end{equation}
The estimator $\widetilde \lambda_n^{\mathrm{med}}$ is very simple  and quick to evaluate. It is a reasonable procedure as it follows standard ideas from aggregated estimators (see e.g. \cite{lavancier:rochet:16}). Let us add that it requires only to tune a grid of reasonable values for the number of block cells. To set this grid, we suggest to start with a small number of blocks, 9 or 16, and increase it until the estimate $\widehat \lambda_n^\mathrm{med}$ significantly deviates from the other ones.

Tables~\ref{tab:nothing}-\ref{tab:delete} summarize the results. We report empirical results for $\widehat\lambda_n^\mathrm{std}$, for the median-based estimator $\widehat\lambda_n^\mathrm{med}$ for $k_n=9, 25$ and $49$ and for the data-driven estimator $\widetilde\lambda_n^\mathrm{med}$.
Table~\ref{tab:nothing} reports empirical means and standard deviations for the pure case (A). Tables~\ref{tab:add} and~\ref{tab:delete} are respectively related to the settings (B) and (C). The two latter can affect significantly the bias of the estimator. In both tables, we report the bias of the different estimators and the gain (in percent) in terms of mean squared error of $\widehat\lambda=\widehat\lambda_n^{\mathrm{med}}$ or  $\widehat\lambda=\widetilde\lambda_n^{\mathrm{med}}$ with respect to $\widehat \lambda_n^\mathrm{std}$, i.e. for each model and each value of $\rho, n, k_n$, we  compute
\begin{equation}
    \label{eq:gain}
    \widehat{\mathrm{Gain}} ( \widehat\lambda )=  \left( \frac{\widehat{\mathrm{MSE}}(\widehat\lambda_n^{\mathrm{std}}) - 
  \widehat{\mathrm{MSE}}( \widehat\lambda )}{\widehat{\mathrm{MSE}}(\widehat\lambda_n^{\mathrm{std}})} \right) \times 100\%
  \end{equation}  
 where $\widehat{\mathrm{MSE}}$ is the empirical mean squared error based on the 500 replications. Thus a positive (resp. negative) empirical gain means that the median-based estimator is more efficient (resp. less efficient) than the standard procedure.

Table~\ref{tab:nothing} shows that the standard and the median-based estimators are consistent when $n$ increases. The parameter $k_n$ looks crucial when $n=1$. In particular the bias seems to increase with $k_n$. When $n=2$, its  influence is much less important. It is also interesting to note that the choice of $k_n$ does not change that much the empirical standard deviations. With absence of outliers, the standard estimator obviously outperforms the median-based estimators, but it is interesting to note that the loss of efficiency is not too important as $n$ increases. Our data-driven median-based estimator, surprisingly, exhibits very nice properties. The bias seems to be averaged over the $k_n$ when $n=1$ and the procedure is even able to reduce the standard deviation. Following the remark after Proposition~\ref{prop:median} the standard deviation of the median-based estimator is difficult to estimate but it can be bounded by the corresponding standard deviation under the Poisson case. According to the simulation setting, this upper-bound is equal to $\sqrt{\pi \lambda /2}/(2n)$, which is equal to $4.4$ when $n=1$ and $2.2$ when $n=2$. This indeed bounds the empirical standard deviation of $\widetilde \lambda_n^\mathrm{med}$ and $\widehat \lambda_n^\mathrm{med}$ for any $k_n$.

Regarding Tables~\ref{tab:add} and~\ref{tab:delete}, we can observe that $\widehat\lambda_n^{\mathrm{std}}$ gets biased. 
As expected, this bias is less important for $\widehat \lambda_n^{\mathrm{med}}$. 
When $n=1$ and $\rho=0.05$, the standard estimator remains much better than the median-based estimator: the gain is negative and can reach very low values. This behaviour also holds in the case (B) when $\rho=0.1$. In the other situations, the median-based estimator outperforms the standard procedure with a positive gain for almost all the values of $k_n$. The fluctuation of the gain with $k_n$ is not very satisfactory and justifies again the introduction of a data-driven procedure. In the setting (B), like $\widehat\lambda_n^\mathrm{med}$, $\widetilde \lambda_n^\mathrm{med}$ outperforms the standard estimator when $n=2$ and behaves similarly to the standard estimator when $n=1$ and $\rho=0.1$. Like in the setting (A), the standard deviation of $\widetilde\lambda_n^\mathrm{med}$ is shown to be smaller than the ones of $\widehat\lambda_n^\mathrm{med}$ for all values of $k_n$, which explains why we observe a higher gain. The conclusion for the setting (C) is unambiguous: the performances of $\widetilde\lambda_n^\mathrm{med}$ are very good even for small observation window or when only 5\% of points on average are deleted. Again, the gain of $\widetilde\lambda_n^\mathrm{med}$ is larger than the gains obtained from $\widehat\lambda_n^\mathrm{med}$ for all the values of $k_n$, except for the {\sc dpp2} model when  $n=2$ an $k_n=49$ for which the observed empirical gain is slightly larger. As a general comment for Tables~\ref{tab:add} and~\ref{tab:delete}, we observe the more repulsive the point pattern, the higher the performances of the robust estimates.\\

\begin{table}[H]
\centering
{\small\begin{tabular}{rrrrrr}
  \hline
&  \multicolumn{5}{c}{Empirical mean (Standard Deviation)} \\
& \multicolumn{1}{c}{$\widehat\lambda_n^{\mathrm{std}}$} & \multicolumn{3}{c}{$\widehat \lambda_n^\mathrm{med}$} & $\widetilde \lambda_n^\mathrm{med}$\\
 &  & $k_n=9$  & $25$    & $49$ &\\
  \hline
\multicolumn{3}{l}{ {\sc dpp1} } &&\\
$n=1$ &  49.7 (3.5)&  50.6   (4.3)&   52.1    (4.3)&   54.1   (4.6)& 52.1 (3.9) \\
  $n=2$ &49.5 (1.6)&  49.8 (2.1)&   50.1  (2.1)&   50.7  (2.1)& 50.1  (1.8)\\
&&&&&\\
\multicolumn{3}{l}{ {\sc dpp2}  }&&\\
$n=1$&  50.0 (3.0) & 51.1 (3.7) &  52.8  (3.9)&   55.3 (3.7)& 52.9 (3.4)\\
$n=2$ & 50.0 (1.5) & 50.3 (1.8)  & 50.6  (1.9)&   51.2 (1.9)& 50.6 (1.6)\\
\hline
\end{tabular}
\caption{\label{tab:nothing} Empirical means and standard deviations between brackets of estimates of the intensity $\lambda=50$ for different models of determinantal point processes ({\sc dpp1, dpp2}). The empirical results are based on 500 replications simulated on $[-n,n]^2$ for $n=1,2$. The first column corresponds to the standard estimator given by~\eqref{eq:std}
while the following ones correspond to the median-based estimators given by~\eqref{eq:median} for different number of cells $k_n$ and by~\eqref{eq:medianDD} for the data-driven procedure.}
}
\end{table}

\begin{table}[H]
\centering
\begin{tabular}{rrrrrr} 
  \hline
&  \multicolumn{5}{c}{Bias (Gain of MSE \%)} \\
& \multicolumn{1}{c}{$\widehat\lambda_n^{\mathrm{std}}$} & \multicolumn{3}{c}{$\widehat \lambda_n^\mathrm{med}$} & $\widetilde \lambda_n^\mathrm{med}$\\
 &  & $k_n=9$  & $25$    & $49$ &\\
  \hline
\multicolumn{1}{l}{$\rho=0.05$}  &&&&&\\\
{\sc dpp1}, $n=1$  &1.9 (0)&   1.7 (-48)  &  3.0 (-73)&    5.2 (-203) &3.2 (-62)\\
  $n=2$ &1.9 (0)&   0.4 (18)&    0.8 (12)&    1.4 (-1)& 0.8  (36)\\
  {\sc dpp2}, $n=1$ &2.3 (0) &  2.0   (-39)&    3.5  (-86)&    5.9  (-253)& 3.6 (-72)\\
  $n=2$ &2.3 (0)&   0.8 (36)&    1.1 (30)&    1.8 (10)& 1.2  (42)\\
 \hline
  \multicolumn{1}{l}{$\rho=0.1$}  &&&&&\\\
   {\sc dpp1}, $n=1$ &4.7 (0)& 2.8 (13)& 4.0 (-8)& 6.0 (-61)& 4.2 (1)\\
  $n=2$ &4.8 (0)&   1.1 (71)&    1.3 (70)& 1.9 (66)& 1.4  (77) \\
  {\sc dpp2}, $n=1$ & 5.0   (0)&   3.2 (19)&    4.4 (-4)&  6.8 (-87)& 4.7 (1)\\
  $n=2$ & 4.8 (0)&   1.1 (76)& 1.3 (77)& 2.0 (69)& 1.4  (80)\\
   \hline
\end{tabular}
\caption{\label{tab:add} Bias and empirical gains in percent between brackets, see~\eqref{eq:gain}, for the standard and  median based estimators for different values of $k_n$. The empirical results are based on 500 replications generated on $[-n,n]^2$ for $n=1,2$ for the models {\sc dpp1, dpp2} where $5\%$ or $10\%$ of points are added to each configuration. This corresponds to the case (B) described in details above.}
\end{table}

\begin{table}[H]
\centering
\begin{tabular}{rrrrrrr}
  \hline
&  \multicolumn{5}{c}{Bias (Gain of MSE \%)} \\
& \multicolumn{1}{c}{$\widehat\lambda_n^{\mathrm{std}}$} & \multicolumn{3}{c}{$\widehat \lambda_n^\mathrm{med}$} & $\widetilde \lambda_n^\mathrm{med}$\\
 &  & $k_n=9$  & $25$    & $49$ &\\
  \hline
\multicolumn{1}{l}{$\rho=0.05$}  &&&&&&\\\
{\sc dpp1}, $n=1$  &
-2.8   (0)&  -2.0 (0)&   -0.7 (3)&    1.6    (-10)& -0.6  (30)\\    
  $n=2$ & 
 -2.9 (0)&  -2.8 (-9)&   -1.8 (25)&   -0.6 (52)& -1.7  (40)\\
  {\sc dpp2}, $n=1$ & 
-2.6 (0) & -1.5 (-9)&    0.2  (4)&    2.7    (-51)&  0.2  (28)\\
  $n=2$ & 
-2.6   (0)&  -2.3    (-4)&   -1.1     (41)&   -0.1     (56)& -1.1  (53)\\
  \hline
  \multicolumn{1}{l}{$\rho=0.1$}  &&&&&&\\\
   {\sc dpp1}, $n=1$ &
-5.4   (0)&  -4.7    (-3)&   -2.3     (35)&    0.0     (46)& -2.3  (48)\\
  $n=2$ & 
-5.4   (0) & -5.0    (-1) &  -1.9 (69)&   -1.7  (74)& -2.3  (69)\\
  {\sc dpp2}, $n=1$ &  
-4.9   (0)&  -3.9 (5)&   -0.8 (41)&   2.0 (36)& -0.9  (56)\\
  $n=2$ &
-5.2   (0) & -4.4 (12)&   -1.3 (79)&   -1.1     (83)&-1.6  (80)\\
   \hline
\end{tabular}
\caption{\label{tab:delete} Bias and empirical gains in percent between brackets, see~\eqref{eq:gain}, for the standard and median based estimators for different values of $k_n$. The empirical results are based on 500 replications generated on $[-n,n]^2$ for $n=1,2$ for the models {\sc dpp1, dpp2} where $5\%$ or $10\%$ of points are deleted to each configuration. This corresponds to the case (C) described in details above.}
\end{table}

The differences of performances of the median-based estimators between the settings (B) or (C) were not expected. To investigate this more, we extend the simulation study. For the case (B), we investigate a different number of randomly chosen sub-squares (specifically 1,2 and 4 sub-squares) and with different side-length (specifically $n/10, n/5$ and $2n/5$) into which points are added. For the setting (C), we also investigate the possibility to delete on average $\rho=5\%$ of the initial points in 1, 2 or 4 randomly chosen sub-squares. We consider only the {\sc dpp2} model and the estimator $\widetilde \lambda_n^\mathrm{med}$. Table~\ref{tab:review} reports empirical results based on 500 replications. Like Tables~\ref{tab:add} and~\ref{tab:delete}, we report the empirical bias and gain. For the contamination (B), like Table~\ref{tab:add}, we observe that when $n=1$, the results are not in favor of $\widetilde \lambda_n^\mathrm{med}$. When $n=2$, we remark, as expected, that the larger the sub-squares $\Delta_n$, the lower the gain.  
The differences are quite similar when the number of sub-squares increases: the larger the number of sub-squares, the lower the gain.
As a conclusion, when data exhibit repulsion with a suspicion of areas with extra points, we recommend to use the estimator $\widetilde\lambda_n^\mathrm{med}$ if those areas  are not too large. Regarding the contamination (C), the conclusion is different. Even if we observe that the performances of the estimator decrease with the number of sub-squares, the gain is still very significant. In other words, the estimator $\widetilde\lambda_n^\mathrm{med}$ is shown to be very robust to missing information for repulsive point patterns.

To go further, as suggested by one reviewer, we investigate another type of outliers which is the addition (resp. deletion) of points uniformly on $W$ (resp. on $\bX$). When $5\%$ of points are added on average, we observe empirical biases and gains of $7.7.$ and $-105\%$ when $n=1$ and $5.5$ and $-26\%$ when $n=2$.  This clearly shows the limitation of the median-based estimator for repulsive patterns. It should not be used at all if we think that extra data are uniform on the observation domain. When, we  delete $5\%$ of points uniformly, we observe $-2.5$ and $53\%$ for the empirical bias and gain, when $n=1$ and $-3.4$ and $23\%$ when $n=2$. Surprisingly, the median-based estimator tends to be quite efficient compared to the standard estimator. It is somehow difficult to explain why the empirical bias increases with $n$. Overall, we think that a median-based estimator, is not tailored-made to take into account for outliers which are not "enough" isolated. We believe another approach should be considered for such a problem, like the one proposed by \citet{redenbach:16}. In the mentioned paper, the authors construct an MCMC algorithm which estimates the parameters of the superposition of a Strauss point process and a Poisson point process.\\

\begin{table}[htbp]
\centering
\begin{tabular}{rrrr}
\hline
 &\multicolumn{3}{c}{Number of sub-squares }\\
& 1 & 2 & 4 \\
\hline
Contamination (B)\\
$|\Delta_n|^{1/2}=n/10, n=1$&
4.3    (8)&  5.3  (-25)&  7.0  (-80)\\
$n=2$&
1.3   (81)&  2.2   (68)&  3.6   (39)\\
$|\Delta_n|^{1/2}=n/5, n=1$&
 5.1  (-11)&  6.1  (-55)&  7.2  (-96)\\
$n=2$&
2.0  (72)&  3.2  (48)  &4.6    (7)\\
$|\Delta_n|^{1/2}=2n/5, n=1$&
6.6  (-66) & 7.3 (-93)&  7.5 (-109)\\
$n=2$&
3.7   (35)&  4.9    (1)&  5.1  (-16)\\
\hline
Contamination (C)\\
$n=1$ &-0.9 (63)& -1.4 (59)& -2 (55)\\
$n=2$&-1.0 (78)& -1.3 (71) &-2 (55)\\
\hline
\end{tabular}
\caption{\label{tab:review} Bias and empirical gains in percent between brackets, see~\eqref{eq:gain}, for the estimator $\widetilde\lambda_n^\mathrm{med}$ given by~\eqref{eq:medianDD}. The empirical results are based on 500 replications generated on $[-n,n]^2$ for $n=1,2$ for the model {\sc dpp2}. To each point pattern, we add (contamination (B)) or delete (contamination (C)) on average 5\% of points in 1,2 or 4 randomly chosen sub-squares $\Delta_n$. In the setting (B), we  investigate different values for the side-length of $\Delta_n$.}
\end{table}

\begin{figure}[htbp]
\subfigure[Contamination (B), one sub-square]{\includegraphics[scale=.5]{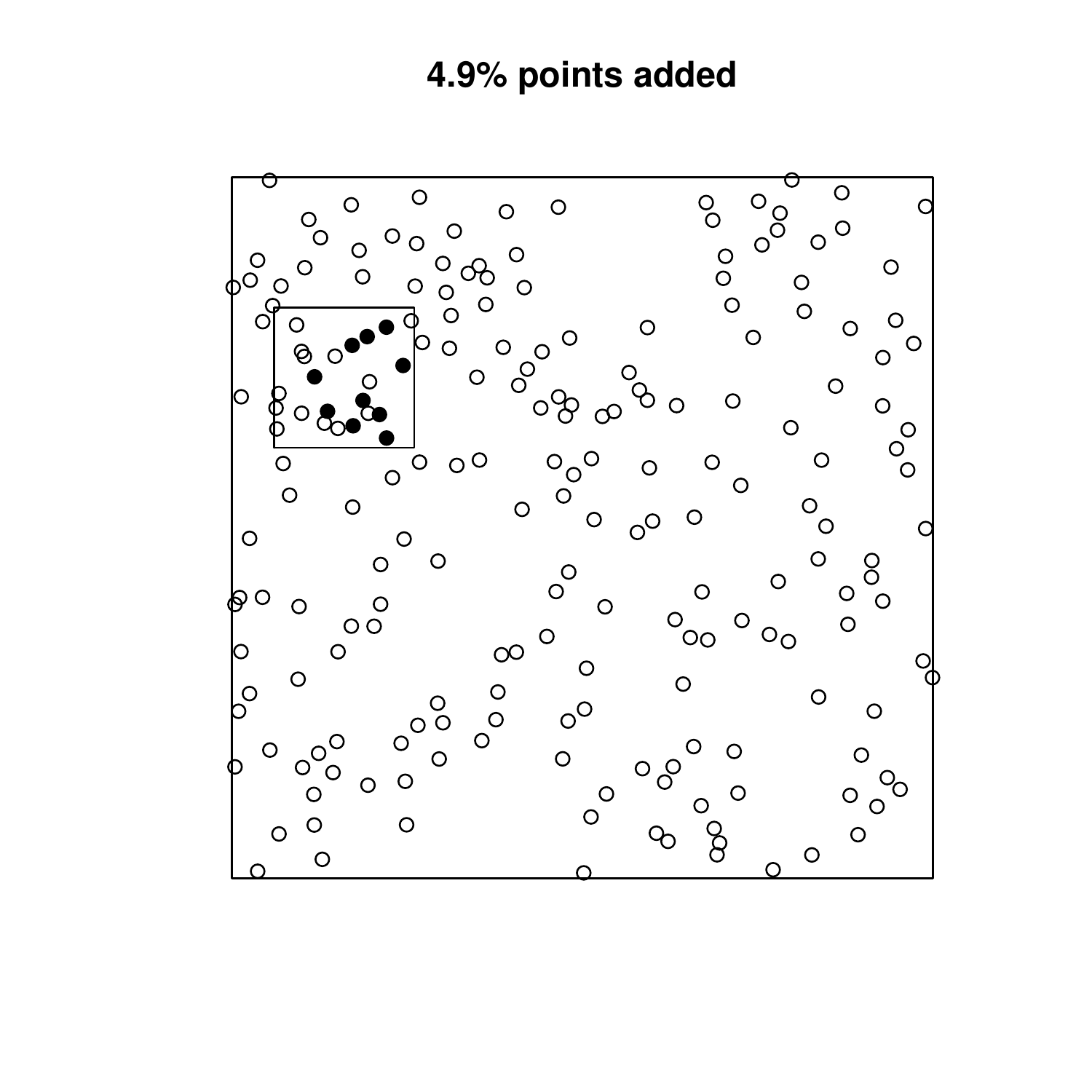}}
\subfigure[Contamination (B), four sub-squares]{\includegraphics[scale=.5]{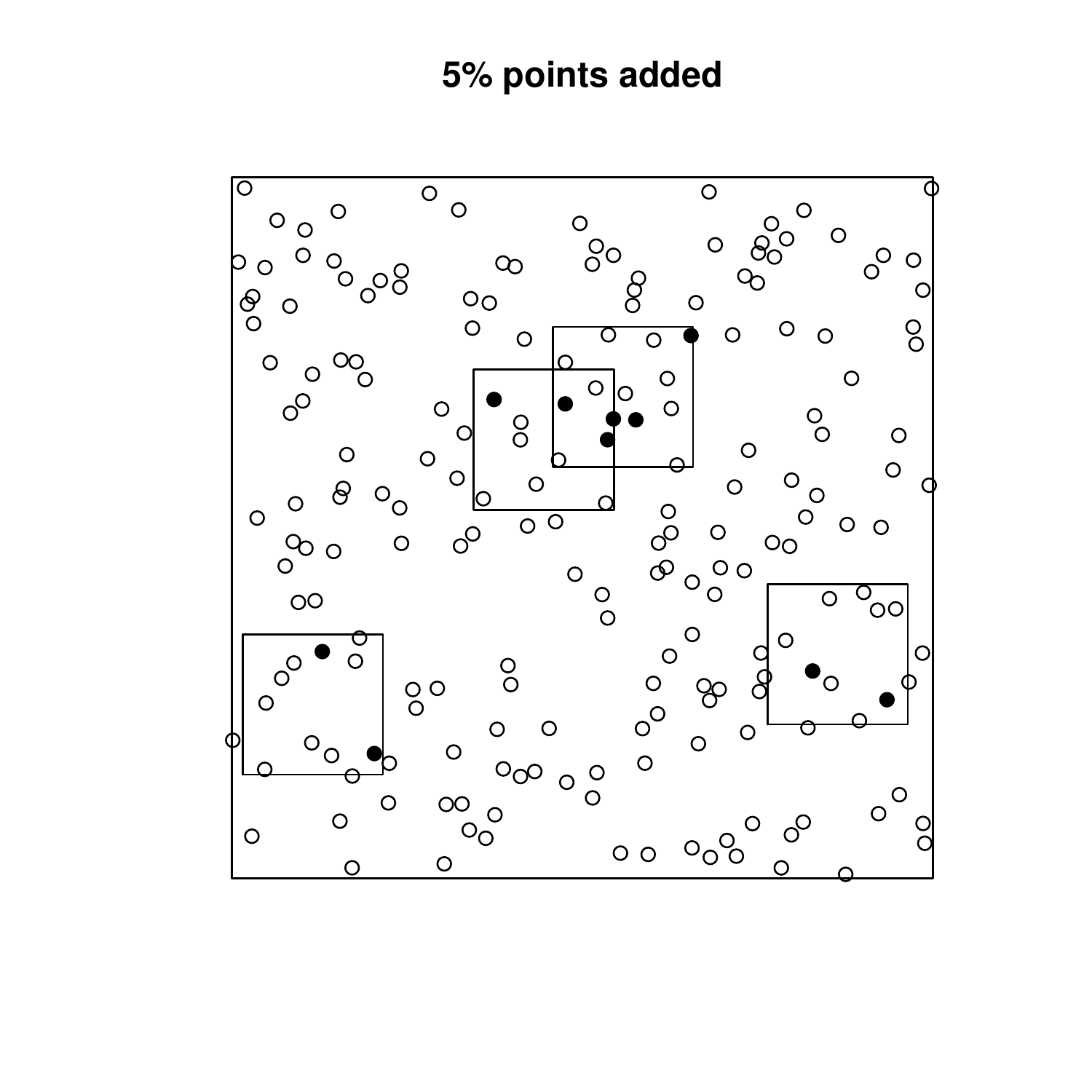}}
\subfigure[Contamination (C), one sub-square]{\includegraphics[scale=.5]{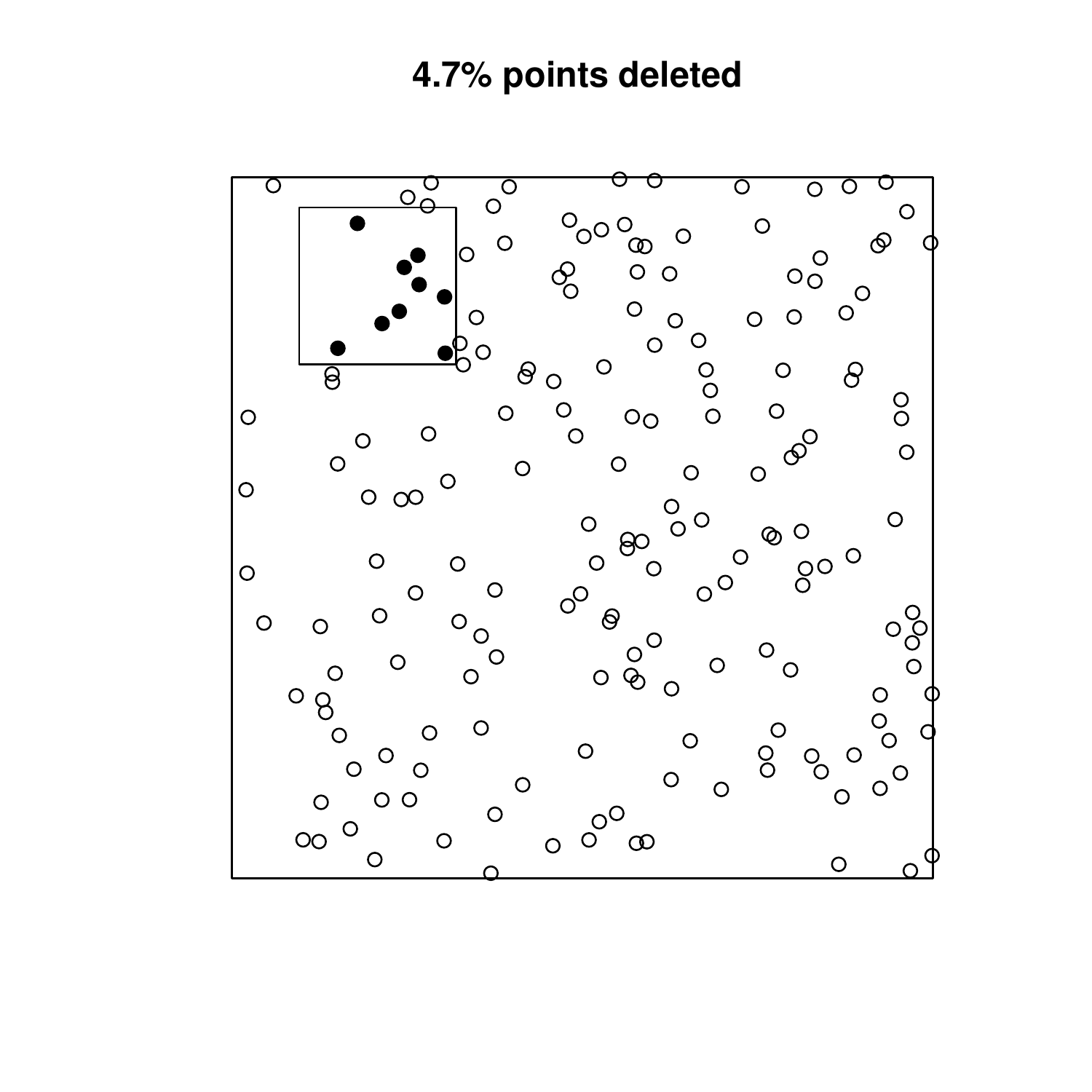}}
\subfigure[Contamination (C), four sub-squares]{\includegraphics[scale=.5]{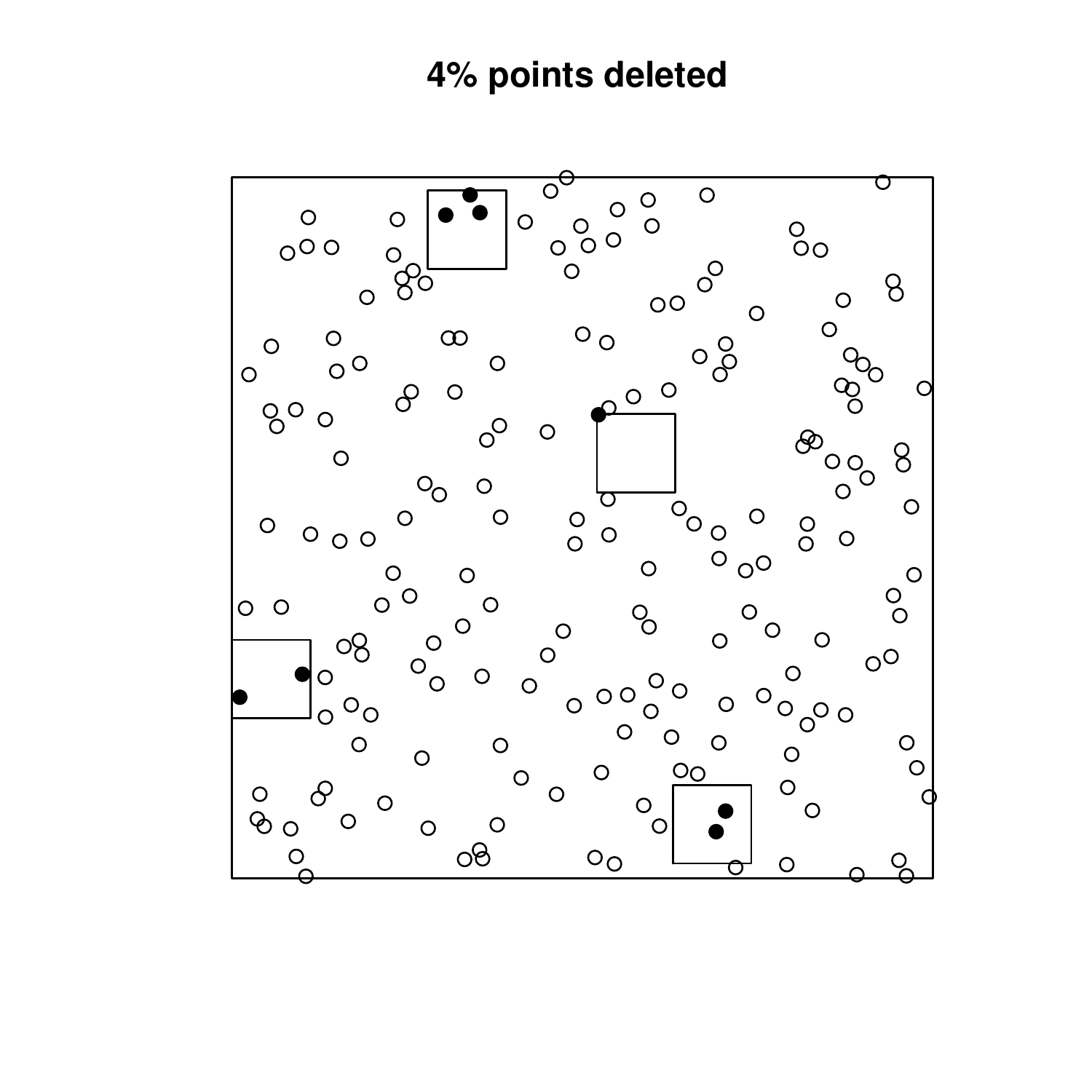}}
\caption{\label{exContamination} Examples of realizations after a contamination of type (B) or (C). The initial patterns are  realizations of the {\sc dpp2} model with intensity $\lambda=50$, on the domain $[-1,1]^2$. For the setting (B) (Figures (a) and (b)), the filled circles represent extra points added to the initial pattern. For the setting (C), the filled circles represent deleted points from the initial pattern. On average, 5\% of points are added or deleted.}
\end{figure}

\section{Conclusion} \label{sec:conclusion}

In this paper, we focus on the class of stationary determinantal point processes and present two estimators of the intensity parameter for which we prove asymptotic properties. Among the two estimators, one of them, namely the median-based estimator is tailored to be robust to outliers. The median-based estimator depends on a tuning estimator, the number of blocks into which the original window is divided. The empirical findings show that the results are quite sensitive to this parameter. To correct that sensitivity we propose a combined approach and define the estimator $\widetilde\lambda_n^\mathrm{med}$ as the median of median-based estimators computed for different number of blocks. The estimator $\widetilde\lambda_n^\mathrm{med}$ is very simple and quick to perform.

As a general conclusion of the simulation study, this combined estimator turns out to be robust to outliers. 
When at least 5\% out of 200 or more points, on average,  lying in possibly different areas of the observation domain are assumed to be not observed, we recommend the use of the combined estimator. If at least 10\% of points are added to a regular pattern with at least 200 points, on average, and when these extra points located in 1,2 or 4 areas, we also recommend to use the combined estimator. This also holds, if 5\% of points are added to a regular pattern with an average of at least 800 points,  or with patterns with at least an average of 200 points for which the extra points are localized in one small area (say 1/25th of the area of the observation domain). In the other situations, we recommend to use the standard estimator.
 
In this work, we did not aim at detecting outliers or detecting areas where problems are suspected (abundance or lack of points). If the assumption of stationarity seems valid, an inspection of the scan statistics (see e.g. \cite{baddeley:rubak;turner:15}) or a large difference between the median-based estimator and the standard estimator of the intensity parameter might allow the user to reconsider the observation window in a second step. This has not been considered in this papper.

\section*{Appendix}

\subsection*{Proof of Proposition~\ref{prop:ZW}}

It can be shown from the spectral decomposition~\eqref{eq:spectral} (see e.g. \cite{lavancier:moller:rubak:15}), that the number of points in $S$ satisfies $N(\bX\cap S) = \sum_{k=1}^{\infty} \mathcal{B}(\beta^S_k)$ in distribution, where $\mathcal{B}(\beta^S_k)$ are independent Bernoulli random variables with parameters $\beta^S_k$.

For $n\geq 1$, we define $N_n(\bX\cap S) = \sum_{k=1}^{n} \mathcal{B}(\beta^S_k)$. Under the assumptions of Proposition~\ref{DPPexistence}, $N_n(\bX\cap S)$ converges in distribution to $N(\bX\cap S)$ as $n$ tends to infinity. 
The random variable $N_n(\bX\cap S)$ is nothing else than a Poisson-Binomial distribution. 

We first prove the first part of~\eqref{eq:d0d1}. By \citet[Theorem 3.4]{zacharovas:hwanf:2010}, we have for all $m\geq 0$,
\begin{align}
\left| \P(N_n(\bX\cap S)) = m) -  \P(\Pi(\theta_{1,n}) = m  )\right| \leq \sqrt{3} (\sqrt{e}-1) \frac{\theta_{2,n}\sqrt{\theta_{1,n}}}{(\theta_{1,n}-\theta_{2,n})^2}
\label{eq:ZW}
\end{align}
where $\theta_{i,n} =\sum_{k=1}^n (\beta_k^S)^i$ for $i=1,2$. As $n \to \infty$, we have
\begin{align}
  \theta_{1,n}&\to \sum_{k\geq 1}\beta_k^S=\E(N(\bX\cap S)) =\lambda |S|\label{eq:th1n}\\
\theta_{2,n}&\to \sum_{k\geq 1}(\beta_k^S)^2=\E(N(\bX\cap S)) -\Var(N(\bX\cap S)) =\int_{S^2}C^2(x-y)\dd x \dd y.\label{eq:th2n}
\end{align}
We note first, that, a change of variables yields 
$\int_{S^2}C^2(x-y)\dd x \dd y < |S| \, \check C_0$ and that, second,
by Parseval's identity and since $0\leq\F(C)< 1$,
 \begin{align*}
  \lambda - \check C_0 = \lambda- \int_{\R^d} \F(C)^2(x)\dd x >  \lambda - \int_{\R^d} \F(C)(x)\dd x = 0
 \end{align*}
whereby we deduce that
\[
  \lim_{n\to \infty} (\theta_{1,n}-\theta_{2,n}) \geq |S| \left(\lambda -\check C_0\right) >0.
\]
Hence, the first part of~\eqref{eq:d0d1} is obtained by letting $n$ tend to infinity in~\eqref{eq:ZW}. For the second part of~\eqref{eq:d0d1}, we use \citet[Theorem~4.2]{zacharovas:hwanf:2010} which states that for  some constant $\tilde c$ independent of $m$, $n$ and $S$
\begin{equation}
  \label{eq:Deltan}
  |\Delta_n(S,m) | \leq \frac{\sqrt{15}(\sqrt{e}-1)}2 \frac{\theta_{2,n}^2 \sqrt{\theta_{1,n}}}{(\theta_{1,n}-\theta_{2,n})^3} +  \tilde c\frac{\theta_{3,n} \sqrt{\theta_{1,n}}}{(\theta_{1,n}-\theta_{2,n})^{5/2}}
\end{equation}
where $\Delta_n(S,m)=\P(N_n(\bX\cap S)=m) - \P(\Pi(\theta_{1,n})=m) \left( 1-\omega(\theta_{1,n},m)  \theta_{2,n}/2 \right)$ and  
$\theta_{3,n}= \sum_{k=1}^n (\beta_k^S)^3$. Since $\beta_k^S\leq 1$, then using~\eqref{eq:th1n} and~\eqref{eq:th2n}, we can show that 
\[
 \lim_{n\to \infty} \frac{\sqrt{15}(\sqrt{e}-1)}2 \frac{\theta_{2,n}^2 \sqrt{\theta_{1,n}}}{(\theta_{1,n}-\theta_{2,n})^3}\to \kappa_1 
  \quad \mbox{ and } \quad 
  \sup_{n\to \infty}\frac{\theta_{3,n} \sqrt{\theta_{1,n}}}{(\theta_{1,n}-\theta_{2,n})^{5/2}} \leq \frac1{|S|} \, \frac{\lambda^{3/2}}{(\lambda -\check C_0)^{5/2}}.
\]
Using these results, we deduce the second part of~\eqref{eq:d0d1} by letting $n\to \infty$ in~\eqref{eq:Deltan}.

\subsection*{Proof of Proposition~\ref{prop:median}}

We start with two ingredients used in (a)-(c). First, from Proposition~\ref{prop:ZW}
\begin{equation} \label{eq:ZW2}
\sup_{m\geq 0} |\P(N(\bX \cap C_{n,0})=m) - \P( \Pi(\lambda c_n)=m  ) | = \mathcal O(c_n^{-1/2}).  
\end{equation}
Second using Stirling's formula we have for any $\omega \in \R$ as $n\to \infty$
\begin{equation} \label{eq:stirling}
  \P( \Pi(\lambda c_n) = \lfloor v_n \rfloor) \sim 
\left\{
\begin{array}{ll}
(2\pi \lambda c_n)^{-1/2} & \mbox{ if } v_n= \lambda c_n + o(c_n^{1/2}) \\
(2\pi \lambda c_n)^{-1/2}e^{-\omega^2/2} & \mbox{ if } v_n= \lambda c_n + \omega c_n^{1/2}.
\end{array}
\right.
\end{equation}

(a) Under the assumptions (i)-(v) described below, \cite{coeurjolly:15}[Theorem 4.3] proved that 
 \begin{equation} \label{eq:thm43}
   \sqrt{|W_n|} s_n \left( \frac{\Mee (\bZ)}{c_n} - \frac{\Me_Z}{c_n}
   \right) \to \mathcal N(0,1/4)
 \end{equation}
 in distribution.

\noindent (i) As $n\to \infty$, $k_n\to \infty, c_n \to \infty$ and $k_n/c_n^{\eta^\prime/2} \to 0$ where $0<\eta^\prime<\eta$ where $\eta$ is given by (iv).\\
\noindent (ii) $\forall t_n= \lambda c_n + \mathcal O(\sqrt{c_n/k_n})$,
  $
     {\P(N(\bX \cap C_{n,0})=\lfloor t_n \rfloor)}/{ \P(N(\bX \cap C_{n,0})=\lfloor \lambda c_n \rfloor)}  \to 1.$\\
(iii)  $\liminf_{n\to \infty} s_n>0$ and $\limsup_{n\to \infty} s_n<\infty$.\\
\noindent (iv) $\bX$ has a pair correlation function $g$ satisfying $\int_{\R^d}|g(w)-1|\dd w<\infty$.\\
\noindent (v) There exists $\eta >0$ such that
\[
  \alpha(m)= \sup_{p\geq 1} \frac{\alpha_{p,p}(m)}{p} = \mathcal O(m^{-d(1+\eta)}) \quad \mbox{ and } \quad
  \alpha_{2,\infty}(m) = \mathcal O(m^{-d(1+\eta)})
\]
where $\alpha_{j,k}(m)$ for $j,k\in \N\cup \{\infty\}$ is defined by~\eqref{def:mixing}.

Therefore, the proof of (a) consists in verifying that $\HW$, $\HCbis$ and $\Hmed$ imply (i)-(v).
By $\HCbis$, $\bX$ is $R$-dependent and thus for any $m>R$ and $j,k\geq 1$, $\alpha_{j,k}(m)=0$. (v) is thus obviously satisfied  and $k_n/c_n^{\eta^\prime/2}\to 0$ can always be fulfilled. From~\eqref{eq:ZW2}
\[
   \frac{\P ( N(\bX \cap C_{n,0})=\lfloor t_n\rfloor) }{ \P(N(\bX \cap C_{n,0})=\lfloor \lambda c_n\rfloor)} 
   \sim \frac{\P(\Pi(\lambda c_n)=\lfloor t_n\rfloor) }{ \P(\Pi(\lambda c_n)=\lfloor \lambda c_n\rfloor)}
\]
as $n\to \infty$. By $\HW$, $t_n=\lambda c_n +o(c_n^{1/2})$ whereby (ii) is deduced from~\eqref{eq:stirling}.  The first part of (iii) corresponds to $\Hmed$ while the second part is deduced from (ii). (iv) is also clearly satisfied since $\bX$ is Brillinger mixing.

(b) First of all, by the Brillinger mixing property, \cite{coeurjolly:15}[Proposition~3.1] can be applied to derive $\Me_Z-\lambda c_n = \mathcal O(\sqrt{c_n})$. Now, under the assumptions (i), (iv) and~(v), we follow the proof of \cite{coeurjolly:15}[Theorem 4.2 Step 1] and state that: $|W_n|^{-1/2}(Z_{n,0}-\lambda c_n) \to \mathcal N(0,\tau^2)$ in distribution for some $\tau>0$, whereby we deduce that $\P(Z_{n,0}\leq\lambda c_n) -1/2\to 0$ as $n\to \infty$. Since $Z_{n,0}$ is a continuous random variable, the latter can also be rewritten as
\[
 \P(Z_{n,0}\leq \lambda c_n)  - \P(Z_{n,0} \leq \Me_Z) =o(1).  
\]
Since $F_Z$ is differentiable with derivative $f_Z(t)=\P(N(\bX \cap C_{n,0})=\lfloor t\rfloor)$, there exists $\tilde M\in [\Me_Z \wedge \lambda c_n,\Me_Z\vee \lambda c_n]$ such that 
\begin{equation}\label{eq:tmp}
  (\Me_Z-\lambda c_n )\P(N(\bX \cap C_{n,0})=\lfloor\tilde M\rfloor)=o(1)
\end{equation}
Since $\Me_Z=\lambda c_n+\mathcal O(\sqrt{c_n})$, there exists $\omega>0$ such that for $n$  sufficiently large, $|\tilde M - \lambda c_n| \leq \omega \sqrt{c_n}$. Using~\eqref{eq:stirling} and the fact that the mode of a Poisson distribution is close to its intensity parameter, we have for $n$ sufficiently large
\begin{align*}
 \P( \Pi(\lambda c_n)=\lfloor \tilde M\rfloor) &\geq  \inf_{ v_n, |v_n-\lambda c_n|\leq \omega\sqrt{c_n}} \P (\Pi(\lambda c_n) = \lfloor v_n \rfloor) \\
 &\geq \frac12  (2\pi \lambda c_n)^{-1/2}e^{-\omega^2/2} \\
 &\geq \frac14 e^{-\omega^2/2} \P(\Pi(\lambda c_n)=\lfloor \lambda c_n \rfloor).
 \end{align*} 
Hence, denoting by $\tilde \omega=e^{-\omega^2/2}/4>0$,  
we deduce using~\eqref{eq:ZW2} that for $n$ sufficiently large
\begin{align*}
\sqrt{c_n} \P ( N(\bX \cap C_{n,0})=\lfloor \tilde M \rfloor)& \geq \frac12 \sqrt{c_n} \P(N(\bX \cap C_{n,0})=\lfloor  \lambda c_n \rfloor) \frac{\P( \Pi(\lambda c_n)=\lfloor \tilde M\rfloor)}{\P( \Pi(\lambda c_n)=\lfloor \lambda c_n\rfloor)}   \\
&\geq \frac{\tilde \omega}2 \; \liminf_{n\to \infty} s_n  =:\check \omega>0
\end{align*}
by Assumption $\Hmed$. Hence, from~\eqref{eq:tmp},  $\check \omega (\Me_Z-\lambda c_n)/\sqrt{c_n}=o(1)$, which yields the result.

(c) From (a)-(b), taking $k_n=o(\sqrt{c_n}/(\Me_Z-\lambda c_n))$ is in agreement with $\HW$. Equation~\eqref{eq:resb} is thus a simple application of Slutsky's lemma.

\section*{Acknowledgements}
The authors would like to thank Denis Allard and Alfred Stein  for giving them the opportunity to contribute to this special issue. The authors are also grateful to the referees and the editor for their valuable comments and suggestions and  to Frédéric Lavancier for fruitful discussions. The research of J.-F. Coeurjolly is  funded by ANR-11-LABX-0025 PERSYVAL-Lab (2011, project OculoNimbus).
The research of Christophe A.N. Biscio are supported by the Danish Council for Independent Research | Natural
Sciences,
grant 12-124675,
"Mathematical and Statistical Analysis of Spatial Data", and by
the "Centre for Stochastic Geometry and Advanced Bioimaging",
funded by grant 8721 from the Villum Foundation.

\bibliographystyle{plainnat}

\bibliography{intensity}

\end{document}